\pdfoutput=1
\documentclass[12pt,a4paper]{amsart}
\usepackage[a4paper,inner=2.5cm,outer=2.5cm,top=2.5cm,bottom=2.5cm]{geometry}
\usepackage{amsmath,amssymb,amsthm,enumerate,mathtools,stmaryrd
}
\usepackage{hyperref}
\hypersetup{colorlinks=true,linkcolor=blue,citecolor=teal,filecolor=magenta,urlcolor=cyan}

\usepackage[T1]{fontenc}
\usepackage[utf8]{inputenc}
\usepackage{makecell}
\usepackage{graphicx}

\usepackage{float}

\usepackage{extarrows}

\usepackage{xcolor}

\newcommand{\C}{\mathbb{C}}

\usepackage{url}

\usepackage[backend=bibtex,style=alphabetic,sorting=nyt,isbn=false,url=false,doi=true,maxalphanames=10,minalphanames=4,mincitenames=4,maxcitenames=10,minnames=4,maxnames=10,giveninits=true,maxbibnames=99]{biblatex}
\theoremstyle{plain}
\newtheorem{theorem}{Theorem}[section]
\newtheorem{proposition}[theorem]{Proposition}
\newtheorem{corollary}[theorem]{Corollary}
\newtheorem{lemma}[theorem]{Lemma}

\theoremstyle{definition}
\newtheorem{definition}[theorem]{Definition}

\makeatletter
\@addtoreset{proofpart}{theorem}
\makeatother
\theoremstyle{remark}
\newtheorem{remark}[theorem]{Remark}

\def\oM{\overline{\mathcal{M}}}

\newcommand{\Z}{\mathbb{Z}}

\newcommand{\cP}{\mathcal{P}}
\newcommand{\cM}{\mathcal{M}}

\newcommand{\cO}{\mathcal{O}}

\newcommand{\res}{\mathop{\rm res}}

\newcommand{\bP}{\mathbb{P}}
\newcommand{\bC}{\mathbb{C}}
\newcommand{\bQ}{\mathbb{Q}}

\newcommand{\SRT}{\mathrm{SRT}}

\newcommand{\restr}[2]{\mathop{\big\lfloor_{{#1}\to {#2}}}}
\newcommand{\set}[1]{\llbracket {#1} \rrbracket}

\DeclareFontFamily{U}{rcjhbltx}{}
\DeclareFontShape{U}{rcjhbltx}{m}{n}{<->rcjhbltx}{}
\DeclareSymbolFont{hebrewletters}{U}{rcjhbltx}{m}{n}

\DeclareMathSymbol{\shin}{\mathord}{hebrewletters}{152}

\DeclareMathOperator{\arcsinh}{arcsinh}

\title{A new spin on polynomial relations among kappa classes}

\author[A.~Alexandrov]{A.~Alexandrov}
\address{A.~A.: Center for Geometry and Physics, Institute for Basic Science (IBS), Pohang 37673, Korea
}
\email{alex@ibs.re.kr}

\author[B.~Bychkov]{B.~Bychkov}
\address{B.~B.: Department of Mathematics, University of Haifa, Mount Carmel, 3498838, Haifa, Israel}
\email{bbychkov@hse.ru}

\author[P.~Dunin-Barkowski]{P.~Dunin-Barkowski}
\address{P.~D.-B.: Faculty of Mathematics, HSE University, Usacheva 6, 119048 Moscow, Russia; HSE--Skoltech International Laboratory of Representation Theory and Mathematical Physics, Skoltech, Bolshoy Boulevard 30 bld. 1, 121205 Moscow, Russia; and NRC “Kurchatov Institute” -- ITEP, 117218 Moscow, Russia}
\email{ptdunin@hse.ru}

\author[M.~Kazarian]{M.~Kazarian}
\address{M.~K.: Faculty of Mathematics, HSE University, Usacheva 6, 119048 Moscow, Russia; and Igor Krichever Center for Advanced Studies, Skoltech, Bolshoy Boulevard 30 bld. 1, 121205 Moscow, Russia}
\email{kazarian@mccme.ru}

\author[S.~Shadrin]{S.~Shadrin}
\address{S.~S.: Korteweg-de Vries Institute for Mathematics, University of Amsterdam, Postbus 94248, 1090GE Amsterdam, The Netherlands}
\email{S.Shadrin@uva.nl}	

\begin{document}
	
\begin{abstract} We prove a recent conjecture of the fourth named author with P. Norbury that states a system of universal polynomial relations among the kappa classes on the moduli spaces of algebraic curves. The proof involves localization and materialization analysis of the spin Gromov-Witten theory of the projective line and is dictated by $\mathbb{Z}_2$-equivariant topological recursion.
\end{abstract}
	
\maketitle
	
\setcounter{tocdepth}{3}
\tableofcontents

\section{Introduction}

We prove a relation among the $\kappa$-classes on the moduli spaces of curves $\oM_{g,n}$, conjectured in~\cite{KN-conjecture}. Let $\oM_{g,n}$ be the moduli space of stable curves of genus $g$ with $n$ marked points, and let $R^*(\oM_{g,n})$ denote its tautological ring.  Define the constants $s_i$, $i=1,2,\dots$ by the formula
\begin{align}\label{eq:J-first}
	\exp(-\textstyle \sum_{i=1}^\infty s_i t^i) = \sum_{i=0}^\infty (-1)^i i! t^i,
\end{align}
where $t$ is a formal variable. Consider the class
\begin{align}\label{eq:J-second}
	\mathbb{J}=1+\sum_{i=1}^{\infty} J_i = \exp(\textstyle \sum_{i=1}^\infty s_i \kappa_i),
\end{align}
where $J_i\in R^i(\oM_{g,n})$ and $\kappa_i$, $i=1,2,\dots$, are the Miller-Morita-Mumford $\kappa$-classes. Equivalently, this class can be defined as follows:
\begin{align}\label{eq:J-class}
	\mathbb{J}=1+\sum_{m=1}^{\infty} \frac{1}{m!} \sum_{a_1,\dots,a_m=1}^\infty \kappa_{a_1,\dots,a_m} \prod_{i=1}^m (-1)^{a_i-1} a_i!,
\end{align}
where $\kappa_{a_1,\dots,a_m}$ are the multi-index $\kappa$-classes obtained from the push-forwards of the $\psi$-classes as $\pi_*(\prod_{i=1}^m \psi_{n+i}^{a_i+1})$, where $\pi \colon \oM_{g,n+m}\to \oM_{g,n}$ is the forgetful map.

The main result of this paper is the following
\begin{theorem} \label{thm:main} We have $J_p=0$ in $R^*(\oM_{g,n})$ for $p>2g-2+n$. Moreover, $J_{2g-2+n}=0$ for $n\geq 2$.
\end{theorem}

This statement was conjectured in~\cite[Conjecture 1, Part 2]{KN-conjecture}. Note that the first part of this conjecture concerns different polynomial relations in $\kappa$-classes and is proven in~\cite{CGG}. In particular, in~\cite[Section 1.1]{KN-conjecture} the following explicit examples of $J_i$ for small $i$ are presented:
\begin{align}
	& J_1 = \kappa_1; \qquad J_2 = \frac 12 (\kappa_1^2-3\kappa_2); \qquad J_3 = \frac{1}{6} (\kappa_1^3 -9\kappa_1\kappa_2 + 26 \kappa_3);
	\\ \notag
	& J_4 = \frac{1}{24} (\kappa_1^4 - 18 \kappa_1^2\kappa_2+27\kappa_2^2 + 104 \kappa_1\kappa_3-426 \kappa_4).
\end{align}
Theorem~\ref{thm:main}, in particular, implies that $J_4$ nontrivially vanishes on $\oM_{1,4}$,   $\oM_{2,1}$, and $\oM_{2,2}$.

Furthermore, let $\lambda_1,\dots,\lambda_g$ be the Chern classes of the Hodge bundle on $\oM_{g,n}$. We prove the following expression for $J_{2g-2+n}$ in the exceptional cases $n\le1$:
\begin{theorem}\label{thm:n10} We have $J_{2g-1}=(-1)^{g-1}\lambda_g\lambda_{g-1}$ in $R^*(\oM_{g,1})$ and $J_{2g-2}=(-1)^{g-2}\lambda_g\lambda_{g-2}$ in $R^*(\oM_{g,0})$.
\end{theorem}
This statement was conjectured in~\cite[Conjecture 2]{KN-conjecture}.

The methods of the proof come from the analysis of the spin Gromov-Witten theory of $(\bP^1,\cO(-1))$ along the lines of~\cite{giacchetto2024spingromovwittenhurwitzcorrespondencemathbbp1} via localization and materialization, and the arguments that we use come from a statement that relates the stationary sector of the spin Gromov-Witten theory of $(\bP^1,\cO(-1))$ to topological recursion~\cite{EO-TR}, more precisely, a $\mathbb{Z}_2$-equivariant version of topological recursion in the sense of~\cite{GKL}. To this end, we have the following statements.

Firstly, we connect the classes $J_p$ to the stationary sector of the spin Gromov-Witten theory of $(\bP^1,\cO(-1))$. We refer the reader to the basic setup of this theory to~\cite{giacchetto2024spingromovwittenhurwitzcorrespondencemathbbp1} that matches the notation that we use here, as well as~\cite{KiemLi}. Let $[\oM_{g,n}(\bP^1,d)]^{\mathrm{loc},\cO(-1)}\in A_{g-1+d+n}(\oM_{g,n}(\bP^1,d))$, $d\geq 0$, be the corresponding localized virtual cycle and consider the class $C_{g,n,d}$ defined as
\begin{align}\label{eq:Cgnd-definition}
	C_{g,n,d} \coloneqq [\oM_{g,n}(\bP^1,d)]^{\mathrm{loc},\cO(-1)} \cdot \prod_{i=1}^n ev_i^*([\mathrm{pt}]).
\end{align}
Let $p\colon \oM_{g,n}(\bP^1,d) \to \oM_{g,n}$, $2g-2+n>0$, be the natural forgetful morphism. We have:
\begin{theorem} \label{thm:C=J} For any $g\geq 0$, $n\geq 0$, $2g-2+n>0$, $d\geq 0$ the Poincar\'e dual of $(-1)^{g+d+1} p_*C_{g,n,d}$ is equal to $J_{2g-2+n-d}$ (in $R^*(\oM_{g,n})$).
\end{theorem}

Secondly, the stationary sector of the descendant spin Gromov-Witten theory of $(\bP^1,\cO(-1))$ whose correlators are defined as
\begin{align} \label{eq:stationarysector}
	 \langle \tau_{k_1}\dots\tau_{k_n}\rangle^{\bP^1,\cO(-1),{\rm des}}_{g} \coloneqq
 (-1)^{g+1} \sum_{d=0}^\infty Q^d\int_{C_{g,n,d}} \prod_{i=1}^n \psi_i^{k_i}
\end{align}
can be computed by $\mathbb{Z}_2$-equivariant topological recursion:
\begin{theorem} \label{thm:TR-spin}
	Let $\omega^{(g)}_n$ be the system of differentials produced via $\mathbb{Z}_2$-equivariant topological recursion on the curve $\bP^1$ with the following input data
%	\begin{align}
%		x& =z-\frac{Q}{2z}; & y&= \log z; & B & =\frac{dz_1dz_2}{(z_1-z_2)^2} + \frac{d\iota(z_1)dz_2}{(\iota(z_1)-z_2)^2};
%		& \iota\colon z\to \frac{Q}{2z},
%	\end{align}
%	or, equivalently, via a different rational parametrization of the spectal curve,
	\begin{align}
		x& =\sqrt{-Q/2}
		\Bigl(z+\frac{1}{z}\Bigr); & y&= 2 \log z; & B & =\frac{dz_1dz_2}{(z_1-z_2)^2} + \frac{d\iota(z_1)dz_2}{(\iota(z_1)-z_2)^2};
		& \iota\colon z\to -\frac{1}{z},
	\end{align}
	where $z$ is a global rational coordinate, and $\iota$ is the involution generating the $\mathbb{Z}_2$-action. Then for $2g-2+n\geq 0$ near $z=0$ in the local coordinate $x^{-1}$ we have:
	\begin{align}
		\omega^{(g)}_n = 2^{2g-2+n} \sum_{k_1,\dots,k_n} \langle \tau_{k_1}\dots\tau_{k_n}\rangle^{\bP^1,\cO(-1),{\rm des}}_{g}
     \prod_{i=1}^n d \frac {(2k_i-1)!!}{x_i^{2k_i+1}} +\delta_{g,0}\delta_{n,2} \frac{dx_1dx_2}{(x_1-x_2)^2}.
	\end{align}
\end{theorem}

This theorem might be considered as a spin version of the Norbury-Scott conjecture~\cite{NS} proved in~\cite{DOSS}, see also~\cite{Eynard-Intersection,FLZ}.

\subsection{Organization of the paper} In Section~2 we recall a particular localization formula for the $\bC^*$-equivariant spin Gromov-Witten theory of $\bP^1$. Then we first analyze its numerical consequences that give an expression for the stationary sector of the ancestor potential of the $\bC^*$-equivariant spin Gromov-Witten theory of $\bP^1$ in terms of the Kontsevich-Witten tau function with shifted times. Repeating this analysis on the level of classes (which is a version of Givental's materialization) we derive tautological relations that prove Theorems~\ref{thm:main} and~\ref{thm:n10}, as well as Theorem~\ref{thm:C=J}.

In Section~3 we explain the origin of this computation: we recall a topological recursion reformulation of the numerical version of~\cite[Conjecture 1, Part 2]{KN-conjecture} given in~\cite[Section 2.2.1]{KN-conjecture} and reformulate it in terms of $\mathbb{Z}_2$-equivariant topological recursion. The latter recursion  appears to be the one given in Theorem~\ref{thm:TR-spin}, and we conclude by proving this theorem.

\subsection{Notation} Throughout the text we use the following notation:
\begin{itemize}
	\item $\set{n}$ denotes $\{1,\dots,n\}$.
	%\item $z_I$ denotes $\{z_i\}_{i\in I}$ for $I\subseteq \set{n}$.
	\item $[u^d]$ denotes the operator that extracts the corresponding coefficient from the whole expression to the right of it, that is, $[u^d]\sum_{i=-\infty}^\infty a_iu^i \coloneqq a_d$.
	\item
	$\restr{u}{v}$ denotes the operator of substitution (or restriction), that is, $\restr{u}{v} f(u) \coloneqq f(v)$.
	%\item $\cS(u)$ denotes $u^{-1}(e^{u/2} - e^{-u/2})$.
\end{itemize}

\subsection{Acknowledgments}  The authors thank Xavier Blot and Reinier Kramer for very useful discussions. We also thank the anonymous referees for useful remarks.

A.~A. was supported by the Institute for Basic Science (IBS-R003-D1).
B.~B. was supported by the ISF Grant 2848/25. B.~B., P.~D.-B., and M.~K. were supported by the Russian Science Foundation (grant No. 24-11-00366).
S.~S. was supported by the Dutch Research Council (OCENW.M.21.233).

The Center of Geometry and Physics of the Institute of Basic Studies, located in Pohang, is acknowledged for hospitality.

\section{Spin Gromov-Witten theory of the projective line} \label{sec:LogTR}

Consider the $\bC^*$ action on $\bP^1$ with two fixed points, $0$ and $\infty$, with the tangent weights $-1$ and $1$, respectively. Consider the moduli space of stable maps $\oM_{g,n}(\bP^1,d)$, $g,n\geq 0$, $d\geq 1$. Let $\pi\colon\overline{\mathcal{C}}_{g,n}(\bP^1,d) \to \oM_{g,n}(\bP^1,d)$ be the universal curve and $f\colon\overline{\mathcal{C}}_{g,n}(\bP^1,d) \to\bP^1$ be the universal map. Our main tool is a formula for the $\bC^*$-equivariant class
\begin{equation} \label{eq:MainClass}
	C_{g,n,d}(T)\coloneqq \Big(\prod_{i=1}^n ev_i^*(\mathbf{0})\Big) \cdot c_{g-1+d} (R^1\pi_*f^*\cO(-1)) \cdot [\oM_{g,n}(\bP^1,d)]^{\mathrm{vir}},
\end{equation}
where $\mathbf{0}$ stands for the equivariant Poincar\'e dual of $0$. This formula is
derived in~\cite{faber2004hodgeintegralspartitionmatrices,giacchetto2024spingromovwittenhurwitzcorrespondencemathbbp1} and we recall it in the next subsection. Note that here we have to choose a lift of the $\bC^*$ action to $\cO(-1)$. We do it by requiring that the fiber weights over $0$ and $\infty$ are $1$ and $0$, respectively. This way the localization formula will produce only graphs of star tree type --- this idea goes back to~\cite{manin1994generatingfunctionsalgebraicgeometry} and is efficiently used in~\cite{faber2004hodgeintegralspartitionmatrices,BlotSS}.

\begin{remark} By construction, $C_{g,n,d}(T)\in A_{g-1+d}^{\bC^*} (\oM_{g,n}(\bP^1,d))$ over the ring $H^*_{\bC^*}(\mathrm{pt})\cong \bQ[T]$. It is a polynomial in $T$, though the localization formula presents it as a Laurent polynomial. The vanishing of the coefficients of the negative powers of $T$ gives us relations between classes in $A_*(\oM_{g,n}(\bP^1,d))$.
\end{remark}

\begin{remark} The non-equivariant limit of the class $c_{g-1+d} (R^1\pi_*f^*\cO(-1)) \cdot [\oM_{g,n}(\bP^1,d)]^{\mathrm{vir}}$ is equal to the localized virtual cycle $[\oM_{g,n}(\bP^1,d)]^{\mathrm{loc},\cO(-1)}$ of the spin curve $(\bP^1,\cO(-1))$~\cite[Corollary~2.4]{KiemLi}. This, together with Equations~\eqref{eq:Cgnd-definition} and~\eqref{eq:MainClass}, implies that $\restr{T}{0}C_{g,n,d}(T) = C_{g,n,d}$ and the intersection numbers
%\begin{align}
	$\int_{C_{g,n,d}(T)} \prod_{i=1}^n \psi_i^{k_i}$
%\end{align}
compute at $T=0$ the stationary sector of the descendant spin Gromov-Witten theory of $(\bP^1,\cO(-1))$ defined in~\eqref{eq:stationarysector}.
\end{remark}

\subsection{The formula}

We fix integers $g,n\geq 0$ and $d\geq 1$.  We denote by $\SRT_{g,n}(\bP^1,d)$ the set of  star rooted trees of genus $g$ with $n$ legs, with edges of total degree $d$, which are the graphs
\begin{align}
	\Gamma=(V,H,\iota \colon H\to H, H^{\iota}\simeq \{\sigma_1,\ldots, \sigma_{n}\}, v_0\in V, d\colon E\to \mathbb{Z}_{\geq 1},
	%k\colon H\setminus H^\iota\to \mathbb{Z},
	o\colon E\to \set{E}),
\end{align}
where $V=V(\Gamma)$ is the set of vertices, $H=H(\Gamma)$ is the set of half-edges, $\iota$ is the involution on $H$ whose cycles of length two form the set of edges $E=E(\Gamma)$ and whose fixed points are the $n$ legs labeled from $1$ to $n$. Furthermore, $g\colon V\to \mathbb{Z}_{\geq 0}$ is the genus function such that the total arithmetic genus of the graph is $g$,  $\sum_{e\in E} d(e) = d$, and we demand that
\begin{enumerate}
	\item the graph $(V,E)$ must be a rooted tree, and all edges are between the root $v_0\in V$ and another vertex (hence the term ``star rooted tree'');
	\item the legs $\sigma_{1},\dots,\sigma_{n}$ are attached to the root vertex $v_0\in V$.
\end{enumerate}
Furthermore, for each vertex $v$ let $n(v)$ denote the valence of $v$. In particular, $n(v_0) = n+|E|$ and $n(v) = 1$ if $v\not=v_0$. We say that a vertex is stable if $2g(v)-2+n(v)>0$, and unstable otherwise.
%We also demand that
%\begin{enumerate}
%	\setcounter{enumi}{2}
%	\item if $h \in H\setminus H^\iota$ is incident to a stable vertex $v$, then  $k(h)\geq 0$.
%\end{enumerate}

Note that since $\Gamma$ is a tree we have $\sum_{v\in V} g(v)=g$ and since $\Gamma$ is a star rooted tree we have a natural isomorphism $V\setminus \{v_0\} \cong E$. It is convenient to order the edges in $E$ by choosing an isomorphism $o\colon E\to \{1,\dots,|E|\}$, and consider the ordering as a part of the data of a star rooted tree. It implies an ordering on $V\setminus \{v_0\}$, and we denote by $e_i$ the $i$-th edge, by $v_i$ the $i$-th non-root vertex (attached to $e_i$) , and let $d_i$ stand for $d(e_i)$ and $g_i$ for $g(v_i)$, respectively, $i=1,\dots,|E|$. Additionally, let $g_0=g(v_0)$.

%\subsubsection{Classes assigned to a tree}

We want to assign to $\Gamma$ a class in $A_*( \oM_{g,n}(\bP^1,d)) \otimes_{\bQ} \bQ[T,T^{-1}]$. First, we assign to $\Gamma$ a stratum in $\oM_{g,n}(\bP^1,d)$ that consists of stable maps with $1+2|E|$ components, where we associate to $v_0$ a constant map to $0\in \bP^1$ of a curve in $\oM_{g_0,n+|E|}$ and to each $v_i\in V\setminus \{v_0\}$ a constant map to $\infty$ of a curve in $\oM_{g_i,1}$ (the unstable vertices are kept for the bookkeeping reasons, we discuss their contraction below). To each edge $e_i$ we associate a Galois covering of $\bP^1$ of degree $d_i$. Let $B_\Gamma$ be the boundary map from $\oM_{g_0,n+|E|}\times \prod_{i=1}^{|E|} \oM_{g_i,1}$ to the corresponding stratum in $\oM_{g,n}(\bP^1,d)$

The $n(v_0)=n+|E|$ marked points on a curve assigned to the root vertex match the $n$ legs and $E$ edges attached to $v_0$ such that the first $n$ marked points correspond to the legs and the last $|E|$ to the edges, respectively, preserving their given ordering.

%Furthermore, we equip $\oM_{g_0,n+|E|}$ associated to the root vertex with the class
%\begin{align}
% \prod_{i=1}^{|E|} \frac{1}{1+d_i \psi_{n+i} T^{-1}}
%\end{align}
%and each $\oM_{g_i,1}$ associated to $v_i$, $i=1,\dots,|E|$, with the class
%\begin{align}
%	\frac{(-T)^{-g_i}\lambda_{g_i} \Lambda(-T^{-1})}{1-d_i \psi_1 T^{-1}}.
%\end{align}

\begin{proposition}[See~\cite{faber2004hodgeintegralspartitionmatrices,giacchetto2024spingromovwittenhurwitzcorrespondencemathbbp1}] For $d\geq 1$ we have:
\begin{multline}\label{eq:C(T)-formula}
	C_{g,n,d}(T) = (-1)^{g+1}\; T^{2g-2+n-d}
	\\
\times\sum_{\Gamma\in \SRT_{g,n}(\bP^1,d)} \frac{(-1)^{|E|}}{|E|!}  (B_\Gamma)_* \left(  \prod_{i=1}^{|E|}  \frac{d_i^{d_i+1}}{d_i! T} \frac{1}{1-d_i \frac{\psi_{n+i}} T} \otimes \bigotimes_{i=1}^{|E|}  	\frac{\frac{(-1)^{g_i}\lambda_{g_i}}{T^{g_i}} \Lambda(\frac{1}{T})}{1+d_i \frac{\psi_1}T} \right).
\end{multline}
\end{proposition}
Here $\Lambda(u)=\sum_{k=0} u^k \lambda_k$. Equation~\eqref{eq:C(T)-formula} is obtained as a product of~\cite[Equations 3.13 and 3.14]{giacchetto2024spingromovwittenhurwitzcorrespondencemathbbp1}, with a specialization of the parameter $\alpha$ in~\emph{op.~cit.} to $-1$ and suitable adjustment of the signs.

\begin{remark} \label{rem:conventions-unstable} This formula involves in some cases the non-existing moduli spaces $\oM_{0,2}$ (for the root vertex in the case $n=1$, $g_0=0$, $|E|=1$ or $n=0$, $g_0=0$, $|E|=2$) and $\oM_{0,1}$ (for the non-root vertices $v$ such that $g(v)=0$ and $n(v)=1$ or the root vertex in the case $n=0$, $g_0=0$, $|E|=1$).
These cases should be considered as the bookkeeping devices with the following conventions:
\begin{align}\label{eq:unstable-intersections}
	\int_{\oM_{0,1}} \frac{1}{1-a\psi_1} & = a^{-2}; &
	\int_{\oM_{0,2}} \frac{1}{1-a\psi_1} \frac{1}{1-b\psi_2} &= (a+b)^{-1}.
\end{align}
\end{remark}

\subsection{Ancestor correlators:~numerical evidence} \label{sec:numerical}
Let $p\colon \oM_{g,n}(\bP^1,d) \to \oM_{g,n}$. Denote by $C_{g,n}(T,Q)$ the (Poincar\'e dual of the) class
\begin{equation}\label{eq:Cgn(T,Q)}
C_{g,n}(T,Q)\coloneqq (-1)^{g+1} \sum_{d=1}^\infty Q^d p_* C_{g,n,d}(T)\in R^*(\oM_{g,n}) \otimes_{\bQ} \bQ[T,Q].
\end{equation}
In this section, we compute the correlators of the ancestors of $\mathbf{0}$ (in the sense of Givental, cf.~\cite{Givental-semisimple,Giv-Quant}) in the $\bC^*$-equivariant extension of the spin Gromov-Witten theory of $(\bP^1,\cO(-1))$ given by~\eqref{eq:Cgn(T,Q)} :
\begin{align}
\langle \tau_{k_1}\dots\tau_{k_n}\rangle^{\bP^1,\cO(-1),\C^*,{\rm anc}}_{g}
\coloneqq \int_{\oM_{g,n}}C_{g,n}(T,Q)\;\psi_1^{k_1}\dots\psi_n^{k_n}.
\end{align}
The general theory implies that this quantity is a polynomial in~$T$ and~$Q$ of homogeneous degree $d_{\rm tot}=\sum k_i-g+1$. In particular, it is identically equal to zero if $d_{\rm tot}<0$. The regular dependence on
$Q$ is evident from the very form of the expression. On the other hand, its regular dependence in $T$ (that is, the fact that the contributions of all summands with $d>d_{\rm tot}$ cancel out), is a manifestation of the regularity in $T$ of the involved equivariant class. Below, we compute these correlators in a closed form that allows one to take a limit as $T\to 0$. In the next section, we reproduce the same computations directly for the classes $C_{g,n}(T,Q)$ themselves. When we pass to $T\to 0$ limit, we routinely add the $d=0$ terms to the definition of the ancestor correlators, that is, we define
\begin{align}
	\langle \tau_{k_1}\dots\tau_{k_n}\rangle^{\bP^1,\cO(-1),{\rm anc}}_{g}
	\coloneqq  (-1)^{g+1}\int_{\oM_{g,n}} \sum_{d=0}^\infty Q^d p_*C_{g,n,d}\; \psi_1^{k_1}\dots\psi_n^{k_n}.
\end{align}
Let us collect these correlators to the potentials
\begin{align}
F^{\bP^1,\cO(-1),\C^*,{\rm anc}}(t_0,t_1,\dots;\hbar)& \coloneqq\sum_{g,n}\frac{\hbar^{2g-2+n}}{n!}
\sum_{k_1,\dots,k_n\ge0} \langle \tau_{k_1}\dots\tau_{k_n}\rangle^{\bP^1,\cO(-1),\C^*,{\rm anc}}_{g}t_{k_1}\dots t_{k_n};
\\ \notag
F^{\bP^1,\cO(-1),{\rm anc}}(t_0,t_1,\dots;\hbar)& \coloneqq\sum_{g,n}\frac{\hbar^{2g-2+n}}{n!}
\sum_{k_1,\dots,k_n\ge0} \langle \tau_{k_1}\dots\tau_{k_n}\rangle^{\bP^1,\cO(-1),{\rm anc}}_{g}t_{k_1}\dots t_{k_n}.
\end{align}
Introduce the rational functions
\begin{equation}\label{eq:Vk}
\begin{gathered}
V_k=\Bigl(\frac{Q}{T-Q}\frac{d}{dQ}\Bigr)^k\frac{Q}{T},\quad k\ge0,
\\V_0=\frac{Q}{T},\quad V_1=\frac{Q}{T(T-Q)}=\frac{1}{T-Q}-\frac{1}{T},
\quad V_2=\frac{Q}{(T-Q)^3},\quad\dots.
\end{gathered}
\end{equation}
Note that $V_k\in\mathbb{Q}[T,(T-Q)^{-1}]$ for $k\ge2$. Define
\begin{align}\label{eq:hatPk}
P_{k}(\hbar)&=-\delta_{k\ge2}V_{k}
-\sum_{g_1=1}^\infty\hbar^{2g_1}\sum_{\substack{a+l=2g_1-2\\a,l\ge0}}
(-1)^{g_1+l}T^{g_1-a} V_{k+l+2}\int_{\oM_{g_1,1}}\lambda_{g_1}\lambda_a\psi_1^{l}.
\end{align}
By construction, the coefficient of any power of $\hbar$ in~$P_{k}(\hbar)$ is an element of $\mathbb{Q}[T,(T-Q)^{-1}]$. % Besides, $P_0(\hbar)$ and $P_1(\hbar)$ have no free terms in the expansion in~$\hbar$.
The following Proposition expresses the ancestor potential in terms of the Kontsevich-Witten potential
\begin{equation} \label{eq:WK-def}
F^{\rm KW}(t_0,t_1,\dots;\hbar)=\sum_{g,n}\frac{\hbar^{2g-2+n}}{n!}
\sum_{k_1,\dots,k_n\ge0} \int_{\oM_{g,n}}\psi_1^{k_1}\dots\psi_n^{k_n}\;t_{k_1}\dots t_{k_n}.
\end{equation}

\begin{proposition} \label{prop:ancGW-KW}
The following formula holds true
\begin{multline}\label{eq:TQshifts}
F^{\bP^1,\cO(-1),\C^*,{\rm anc}}(t_0,t_1,\dots;\hbar)%\\
\approx F^{\rm KW}\bigl(t_0+\tfrac{P_0(\hbar)}{\hbar},t_1+\tfrac{P_1(\hbar)}{\hbar},
t_2+\tfrac{P_2(\hbar)}{\hbar},\dots;(T-Q)\hbar),%-F^{\rm KW}\bigl(t_0,t_1,t_2,\dots;T\hbar),
\end{multline}
with the following reservation: the symbol $\approx$ here and below means that the equality holds up to unstable terms with $2g-2+n\le0$,
and also the terms with $d=0$ or $n\le1$ and arbitrary~$g$.
\end{proposition}

The exceptional terms not covered by Proposition~\ref{prop:ancGW-KW} are also known; they are treated separately, see Section~\ref{sec:geom1} below. Note that the shifts of times $t_k$ to finite values in the Kontsevich-Witten potential are well defined for $k\ge2$ only. In our case the shifts involve $\hbar$, and since the expansions of $P_0/\hbar$ and $P_1/\hbar$ contain strictly positive powers of $\hbar$ only, the shifts of $t_0$ and $t_1$ are also well defined if treated in the expansion in~$\hbar$. This formula expresses each ancestor correlator as a finite combination of intersection numbers of $\psi$ classes with coefficients in $\mathbb{Q}[T,(T-Q)^{\pm1}]$. The geometrical origin of the potential implies that the poles at $T=Q$ cancel out and the coefficients of the potential are actually polynomial in~$T$ and~$Q$. Taking the limit at $T\to 0$ we get (after including the $d=0$ terms):
\begin{align}
& P_0(\hbar)\bigm|_{T=0}=P_1(\hbar)\bigm|_{T=0}=0;\quad
P_{k}(\hbar)\bigm|_{T=0}=(-\partial_Q)^{k-2}\tfrac1{Q^2}=\tfrac{(k-1)!}{Q^k}, \quad k \geq 2;
\\ \label{eq:Anc-WK}  & F^{\bP^1,\cO(-1),{\rm anc}}
=
F^{\rm KW}\bigl(t_0,t_1,
t_2+\tfrac{1!}{\hbar Q^2},t_3+\tfrac{2!}{\hbar Q^3},t_4+\tfrac{3!}{\hbar Q^4},\dots;-Q\hbar).
\end{align}
Here we used that $\lambda_{g_1}^2=0$ for $g_1\geq 1$~\cite{Mumford}, thus the second summand in~\eqref{eq:hatPk} vanishes at $T\to 0$.

As a corollary, we conclude that the right hand side depends regularly on $Q$. This is one of the reformulations of the conjecture in~\cite{KN-conjecture}, see~\cite[Section 2]{KN-conjecture}, where we identify $-Q$ with the parameter $\epsilon$ there. We prove, thereby, (the numerical part of) this conjecture and~\eqref{eq:TQshifts} provides its one-parameter extension.

\begin{remark} Note that formally Equation~\eqref{eq:Anc-WK} holds exactly as it is presented, unlike the statement of Proposition~\ref{prop:ancGW-KW}, where we use the symbol $\approx$. So it is formally not a direct corollary and some additional analysis of the $d=0$ and/or $n\leq1$, and unstable terms is required. We refer for that to Section~\ref{sec:geom1} and~\ref{sec:taut}.
\end{remark}

\begin{remark}
The equality of Proposition holds, following~\eqref{eq:Cgn(T,Q)}, in the series expansion in the positive powers of the variable $Q$. On the other hand, the coefficients on the right hand side in~\eqref{eq:TQshifts} are represented as rational functions in $Q$ and $T$. The specializations $T=0$ for individual terms entering the right hand side of~\eqref{eq:TQshifts} have no cohomological meaning. However, since the coefficients of this function are actually polynomial in $Q$ and $T$, we can compute the non-equivariant limit $T=0$ for the whole function by taking this specialization in the rational coefficients of its particular terms.
\end{remark}

\begin{proof}[Proof of Proposition~\ref{prop:ancGW-KW}]

The mapping $p\colon \oM_{g,n}(\bP^1,d) \to \oM_{g,n}$ maps the stratum in the moduli space of stable maps corresponding to a given star rooted tree~$\Gamma$ to the corresponding boundary stratum in~$\oM_{g,n}$. The curves of this stratum have a ``central'' component of genus~$g_0$ with all marked points situated on this component, and a number other components of genera $g_i$ attached to the central one at double points of the curve and corresponding to the non-root vertices of the tree. In fact, in this description of the stratum in~$\oM_{g,n}$, we should keep only those non-root vertices that correspond to positive genera $g_i$ while those with~$g_i=0$ and corresponding to the non-existing moduli spaces~$\oM_{0,1}$ should be contracted and removed from the description of the stratum; the contribution of such vertices is accounted by just a numerical factor implied by~\eqref{eq:unstable-intersections}.

Next, since all marked points are situated on the central component, the computation of the ancestor correlator is reduced to the computation of intersection numbers of~$\psi$ classes on the moduli space~$\oM_{g_0,n+m}$ corresponding to the central component of the curve where $m$ is the number of remaining attached components, and the contribution of the attached components is accounted by the intersection numbers over the corresponding moduli spaces~$\oM_{g_i,1}$ which are universal quantities independent of a~$\psi$-monomial associated with the marked points.

Putting all these arguments together we arrive at the following expression for the ancestor correlators. Define
\begin{align}\label{eq:hg-coef}
h_{g,k}&=
-\sum_{d=1}^\infty \Bigl(\frac{Q}{T}\Bigr)^d\;\frac{d^{d-1}}{d!}\;\Bigl(\frac{d}{T}\Bigr)^{k+2}\!\!\!\int\limits_{\oM_{g_1,1}}\!\!\!
\frac{(-T)^{g_1}\lambda_{g_1}\Lambda(\frac1T)}{1+\frac{d}{T}\psi_1},\qquad g\ge1,
\\\label{eq:h0-coef}
h_{0,k}&=-\sum_{d=1}^\infty \Bigl(\frac{Q}{T}\Bigr)^d\;\frac{d^{d-1}}{d!}\;\Bigl(\frac{d}{T}\Bigr)^{k},
\\H_g(\psi)&=\sum_{k=0}^\infty h_{g,k}\psi^k.
\end{align}
Then, for $n\ge2$, we get from~\eqref{eq:C(T)-formula}:
\begin{align}\label{eq:anc-first}
\langle \tau_{k_1}\dots\tau_{k_n}\rangle^{\bP^1,\cO(-1),\C^*,{\rm anc}}_{g}
& =\sum_{m,r\ge0}\sum_{\substack{g=g_0+g_1+\dots+g_m\\g_0\ge0,~g_i\ge1,\ i=1,\dots,m}}
\frac{T^{2g_0-2+n+m+r}}{m!\,r!}\times
\\ \notag & \qquad
\int_{\oM_{g_0,m+n}}\psi_1^{k_1}\dots\psi_n^{k_n}\;
({\pi_{r}})_*\Bigl(\prod_{j=1}^mH_{g_i}(\psi_{n+j})\prod_{\ell=1}^r H_{0}(\psi_{n+m+\ell})\Bigr),
\end{align}
where $\pi_{r}\colon\oM_{g_0,n+m+r}\to \oM_{g_0,n+m}$. The restriction $n\ge2$ is imposed to avoid an appearance of summands with $2g_0-2+n+m\le0$. These summands should be treated separately since the central component itself becomes unstable for these summands and requires a further contraction.

Thus, we arrive at the problem of computing the push-forward homomorphism associated with the mapping $\pi_{r}$. The main difficulty in treating integrals of the kind
\begin{align}
& \int_{\oM_{g_0,m+n}}\psi_1^{k_1}\dots\psi_n^{k_n}\;
({\pi_r})_*\Bigl(\prod_{j=1}^mH_{g_i}(\psi_{n+j})\prod_{\ell=1}^r H_{0}(\psi_{n+m+\ell})\Bigr)
\\ \notag & =\int_{\oM_{g_0,m+n+r}}(\pi_r)^*\bigl(\psi_1^{k_1}\dots\psi_n^{k_n}\bigr)
\prod_{j=1}^m H_{g_i}(\psi_{n+j})\prod_{\ell=1}^r H_{0}(\psi_{n+m+\ell})
\end{align}
is the nontriviality of the constant term $h_{0,0}$ in the series $H_0(\psi)$. Indeed, the difference $\pi^*_r\psi_i-\psi_i$ for $i\le m+n$ is annihilated by positive exponents of $\psi_{m+n+1},\dots,\psi_{n+m+r}$, so that in the case $h_{0,0}=0$ we could replace $\pi^*_r\psi_1^{k_1}\dots\psi_n^{k_n}$ in the integral by just $\psi_1^{k_1}\dots\psi_n^{k_n}$ and arrive at the usual intersection number of~$\psi$ classes (over $\oM_{g_0,m+n+r}$).
In the general case we need to get rid of the nontrivial coefficient $h_{0,0}$ by applying a suitable version of the string equation. This can be resolved universally for an arbitrary series of coefficients $h_{g,k}$, not necessarily the ones that we have in~\eqref{eq:hg-coef}--\eqref{eq:h0-coef}, and it can be done with the closed formulas proposed in~\cite{Givental-semisimple}. These closed formulas are nicely summarized in~\cite[Appendix A.3]{Janda-P1}, and the recipe is as follows.

Let us introduce the function~$u=u(h_{0,0},h_{0,1},\dots)$ as the solution of the implicit function equation
\begin{equation}\label{eq:ugen-def}
\frac{u}{T}=\sum_{l=0}^\infty h_{0,l}\frac{u^\ell}{\ell!}
\end{equation}
and denote
\begin{align}
h^*_{g,k}&=\sum_{l=0}^\infty h_{g,k+l}\frac{u^\ell}{\ell!}, && (g,k)\ne(0,0),
\\H_g^{*}(\psi)&=\sum_{k=0}^\infty h^*_{g,k}\,\psi^k, &&  g\ge1,
\\H_0^{*}(\psi)&=\sum_{k=2}^\infty h^*_{0,k}\,\psi^k,
\\\label{eq:Sgen-def}
S&=\frac{T}{1-T h^*_{0,1}}.
\end{align}
Then, by~\cite[Equation (A.5)]{Janda-P1}, with the notations adopted to the present paper, we have
\begin{align}
& \sum_{r=0}^\infty\frac{T^{2g_0-2+n+m+r}}{r!}({\pi_{r}})_*\Bigl(\prod_{j=1}^mH_{g_i}(\psi_{n+j})
   \prod_{\ell=1}^r H_{0}(\psi_{n+m+\ell})\Bigr)
\\ \notag & \qquad =
\sum_{r=0}^\infty\frac{S^{2g_0-2+n+m+r}}{r!}({\pi_{r}})_*\Bigl(\prod_{j=1}^mH^*_{g_i}(\psi_{n+j})
   \prod_{\ell=1}^r H^*_{0}(\psi_{n+m+\ell})\Bigr)
\\ \notag & \qquad =
\prod_{j=1}^mH^*_{g_i}(\psi_{n+j})\;
\sum_{r=0}^\infty\frac{S^{2g_0-2+n+m+r}}{r!}({\pi_{r}})_*\Bigl(
   \prod_{\ell=1}^r H^*_{0}(\psi_{n+m+\ell})\Bigr).
\end{align}
The series $H^*_{0}(\psi)$ has vanishing not only constant but also linear term. This implies by dimension count that the sum on the right hand side is finite.  Elimination of linear term is covered by the dilaton equation and is accounted by the rescaling~\eqref{eq:Sgen-def}. Vanishing of the constant term in~$H^*_0(\psi)$ implies that denoting
\begin{align}
\langle\tau_{k_1}\dots\tau_{k_m}\rangle^{\rm pt}_{g_0}&=\int_{\oM_{g_0,m}}\psi_1^{k_1}\dots\psi_m^{k_m},
\\\label{eq:Pgen-def}
P_{k}(\hbar)&=\delta_{k\ge2}h^*_{0,k}+\sum_{g_1\ge1}\hbar^{2g_1}h^*_{g_1,k}.
\end{align}
we get (for $n\ge2$)
\begin{align}
\langle \tau_{k_1}\dots\tau_{k_n}\rangle^{\bP^1,\cO(-1),\C^*,{\rm anc}}_{g}
& =\sum_{m,r\ge0}\sum_{\substack{g=g_0+g_1+\dots+g_m\\g_0\ge0,\ g_i\ge1,\ i=1,\dots,m}}
\frac{S^{2g_0-2+n+m+r}}{m!\,r!}
\\ \notag & \qquad \times
\int_{\oM_{g_0,m+n+r}}\psi_1^{k_1}\dots\psi_n^{k_n}\;
\prod_{j=1}^mH^*_{g_i}(\psi_{n+j})\prod_{\ell=1}^r H^*_{0}(\psi_{n+m+\ell})
%\\&=\sum_{m\ge0}\sum_{\substack{g=g_0+g_1+\dots+g_m\\g_i\ge0}}
%\frac{S^{2g_0-2+n+m}}{m!}
%\int\limits_{\oM_{g_0,n+m}}\!\!\!\psi_1^{k_1}\dots\psi_n^{k_n}\;
%\prod_{j=1}^mH^*_{g_i}(\psi_{n+j})
\\ \notag &=\sum_{g_0=0}^gS^{2g_0-2+n}\langle\tau_{k_1}\dots\tau_{k_n}[\hbar^{2(g-g_0)}]e^{S\sum_{k=0}^\infty P_{k}(\hbar)\tau_k }\rangle^{\rm pt}_{g_0}.
\end{align}
Thus, finally,
\begin{align}
F^{\bP^1,\cO(-1),\C^*,{\rm anc}}(t_0,t_1,\dots;\hbar)&=
\sum_{g\ge0}\hbar^{2g-2}
\langle e^{\hbar \sum_{k=0}^\infty t_k\tau_k}\rangle^{\bP^1,\cO(-1),\C^*,{\rm anc}}_{g}
\\\notag&\approx
\sum_{g_0\ge0}(S\hbar)^{2g_0-2}
\langle e^{S\hbar \sum_{k=0}^\infty t_k\tau_k}e^{S\sum_{k=0}^\infty P_{k}(\hbar)\tau_k }\rangle^{\rm pt}_{g_0}
%\\\notag&=
%\sum_{g_0\ge0}(S\hbar)^{2g_0-2}
%\langle e^{S\hbar \sum_{k=0}^\infty (t_k-P_{k}(\hbar)/\hbar)\tau_k}\rangle^{\rm pt}_{g_0}
\\\notag&= F^{\rm KW}\bigl(t_0+\tfrac{P_0(\hbar)}{\hbar},t_1+\tfrac{P_1(\hbar)}{\hbar},\dots;S\hbar).
\end{align}

It remains to identify the constants $u,h^*_{g,k},S,P_{k}(\hbar)$ given by~\eqref{eq:ugen-def}--\eqref{eq:Sgen-def} and~~\eqref{eq:Pgen-def} in the case when $h_{g,k}$ are defined by~\eqref{eq:hg-coef}--\eqref{eq:h0-coef}. This can be done explicitly and rational dependence of these constants in~$T$ and~$Q$ is a surprising corollary of these computations.  We have:
\begin{equation}
\sum_{\ell=0}^\infty h_{0,l}\frac{u^\ell}{\ell!}=-\sum_{d=1}^\infty\sum_{l=0}^\infty \Bigl(\frac{Q}{T}\Bigr)^d\;\frac{d^{d-1}}{d!}\;\Bigl(\frac{d}{T}\Bigr)^{\ell}\frac{u^\ell}{\ell!}
=-\sum_{d=1}^\infty\Bigl(\tfrac{Q}{T}e^{\frac{u}{T}}\Bigr)^d\;\frac{d^{d-1}}{d!}.
\end{equation}
Therefore, the implicit equation~\eqref{eq:ugen-def} on~$u$ reads
\begin{equation}
\frac{u}{T}=-\sum_{d=1}^\infty\Bigl(\tfrac{Q}{T}e^{\frac{u}{T}}\Bigr)^d\;\frac{d^{d-1}}{d!}.
\end{equation}
The solution to this equation is given explicitly by $u=-Q$. Indeed, it is well known that the inverse change to that one given by the Lambert curve equation $y=xe^{-x}$ is given by $x=\sum_{d=1}^\infty\frac{d^{d-1}}{d!}y^d$; we just apply this fact to $y=Q/T$.

Next, regarding $X=X(Q)=\tfrac{Q}{T}e^{\frac{u}{T}}=\tfrac{Q}{T}e^{-\frac{Q}{T}}$ as a local change of coordinate we observe the equality $\frac{X}{T}\frac{d}{dX}=\frac{Q}{T-Q}\frac{d}{dQ}$. Therefore, we get
\begin{align}
h^*_{0,k}
& =-\sum_{d=1}^\infty\sum_{l=0}^\infty \Bigl(\frac{Q}{T}\Bigr)^d\;\frac{d^{d-1}}{d!}\;\Bigl(\frac{d}{T}\Bigr)^{k+\ell}\frac{u^\ell}{\ell!}    =-\sum_{d=1}^\infty \frac{d^{d-1}}{d!}\bigl(\tfrac{Q}{T}e^{-\frac{Q}{T}}\bigr)^d\bigl(\tfrac{d}{T}\bigr)^k
\\ \notag & =-\sum_{d=1}^\infty \frac{d^{d-1}}{d!}X^d\bigl(\tfrac{d}{T}\bigr)^k
   =-\Bigl(\frac{X}{T}\frac{d}{dX}\Bigr)^k\sum_{d=1}^\infty \frac{d^{d-1}}{d!}X^d
   =-\Bigl(\frac{Q}{T-Q}\frac{d}{dQ}\Bigr)^k\frac{Q}{T}=-V_k.
\end{align}
As a corollary, we get
\begin{equation}
S=\frac{T}{1-Th^*_{0,1}}=\frac{T}{1+T\,V_1}=T-Q.
\end{equation}
By a similar computation, for $g\ge1$  we get
\begin{align}
h^*_{g,k}&=-\sum_{d=1}^\infty \frac{d^{d-1}}{d!}\bigl(\tfrac{Q}{T}e^{-\frac{Q}{T}}\bigr)^d
\sum_{l\ge0}\bigl(\tfrac{d}{T}\bigr)^{k+2+l}(-1)^\ell
\int_{\oM_{g,1}}(-T)^g\lambda_g\Lambda\bigl(\tfrac1{T}\bigr)\psi_1^\ell
\\ \notag &=-
\sum_{l\ge0}(-1)^{g+\ell}T^g V_{k+l+2}
\int_{\oM_{g,1}}\lambda_g\Lambda\bigl(\tfrac1{T}\bigr)\psi_1^\ell .
\end{align}
Substituting the obtained expressions for $h_k^*$ to~\eqref{eq:Pgen-def} we obtain exactly~\eqref{eq:hatPk}. This completes the proof of Proposition.
\end{proof}

\begin{remark} 
For the \emph{descendant} correlators
\begin{equation}
	\langle \tau_{k_1}\dots\tau_{k_n}\rangle^{\bP^1,\cO(-1),\C^*,{\rm des}}_{g}
	\coloneqq  (-1)^{g+1}\sum_{d=0}^\infty Q^d\int_{\oM_{g,n}(\bP^1,d)}  C_{g,n,d}(T)\; \psi_1^{k_1}\dots\psi_n^{k_n}
\end{equation}
the arguments similar to those in the above proof lead to the following analog of Eq.~\eqref{eq:anc-first}:
\begin{multline}
\langle \tau_{k_1}\dots\tau_{k_n}\rangle^{\bP^1,\cO(-1),\C^*,{\rm des}}_{g}
=\sum_{m,r\ge0}\sum_{\substack{g=g_0+g_1+\dots+g_m\\g_0\ge0,~g_i\ge1~(i\ge1)}}
\frac{T^{2g_0-2+n+m+r}}{m!\,r!}\times
\\
\int_{\oM_{g_0,m+n}}\psi_1^{k_1}\dots\psi_n^{k_n}\;
\prod_{j=1}^mH_{g_i}(\psi_{n+j})\prod_{\ell=1}^r H_{0}(\psi_{n+m+\ell}).
\end{multline}
It implies, in turn, the corresponding expression for the potential
\begin{gather}
F^{\bP^1,\cO(-1),\C^*,{\rm des}}(t_0,t_1,\dots;\hbar)\approx
F^{\rm KW}\bigl(t_0+\tfrac{\widetilde P_0(\hbar)}{\hbar},t_1+\tfrac{\widetilde P_1(\hbar)}{\hbar},\dots;T\hbar),
\\\widetilde P_{k}(\hbar)=-\sum_{d=1}^\infty
\Bigl(\frac{Q}{T}\Bigr)^d\;\frac{d^{d-1}}{d!}\;
\Bigl(\frac{d}{T}\Bigr)^k\biggl(
1+\Bigl(\frac{d}{T}\Bigr)^2
\sum_{g_1=1}^\infty\hbar^{2g_1}
\!\!\!\int\limits_{\oM_{g_1,1}}\!\!\!
\frac{(-T)^{g_1}\lambda_{g_1}\Lambda(\frac1T)}{1+\frac{d}{T}\psi_1}\biggr).
\end{gather}
The descendant correlators are not rational in~$T$ and~$Q$ any more, and it is not obvious at all how to pass to the limit $T\to0$ (even though we know \emph{a priori} that this limit does exist).

The relation between the descendant and the ancestor potentials is given by a linear change of coordinates (with a special treatment of the unstable terms):
\begin{align} \label{eq:ansc-desc}
	F&{}^{\bP^1,\cO(-1),\C^*,{\rm anc}}\approx F^{\bP^1,\cO(-1),\C^*,{\rm des}}\bigm|_{t_k\to\sum_{m=0}^\infty \frac{Q^m}{m!}t_{k+m}}.
\end{align}
\end{remark}

\begin{remark}
Note that both $F^{\bP^1,\cO(-1),\C^*,{\rm des}}$ and $F^{\bP^1,\cO(-1),\C^*,{\rm anc}}$
%, up to controllable corrections,
are solutions of the KdV integrable hierarchy, therefore can be considered in the framework of reductions of both KP and BKP hierarchies.
\end{remark}

\subsection{Materialization} \label{sec:geom1} In this section we compute the class $C_{g,n}(T,Q)\in R^*(\oM_{g,n})\otimes \bQ[T,T^{-1}][[Q]]$ itself. This computation involves the string and the dilaton equation and is known in the realm of localization formulas under the name ``materialization''~\cite{Givental-semisimple,Janda-P1}.
In order to describe the answer, we need an ``$\oM_{g,n}$ version'' of star rooted trees $\SRT_{g,n}$. For completeness, we give a full definition that largely repeats the one given above for $\SRT_{g,n}(\bP^1,d)$.

Let $\SRT_{g,n}$ denote the set of stable graphs of genus $g$ with $n$ legs, which are the graphs
\begin{align}
	\Gamma=(V,H,\iota \colon H\to H, H^{\iota}\simeq \{\sigma_1,\ldots, \sigma_{n}\}, v_0\in V,  o\colon E\to \set{|E|}),
\end{align}
where $V$, $H$, $\iota$, and $E$ are the vertices, half-edges, involution, and edges, as above, and $g\colon V\to \mathbb{Z}_{\geq 0}$ is again the genus function. As before, the graph $(V,E)$ must be a star rooted tree with all legs $\sigma_{1},\dots,\sigma_{n}$ attached to the root vertex $v_0\in V$. The valence of a vertex $v$ is denoted by $n(v)$, and we demand that all vertices are stable, that is, $2g(v)-2+n(v)>0$.

Note that we still have $\sum_{v\in V} g(v)=g$ and the edges $E$ and the non-root vertices $V\setminus \{v_0\}$ are ordered by $o\colon E\to \set{|E|}$; $e_i$, resp. $v_i$ denotes the $i$-th edge, resp., the $i$-th non-root vertex attached to $e_i$. Let $g_i=g(v_i)$, $i=0,\dots,|E|$. Note that the stability condition implies $g_i>0$ for $i\ne0$.

We want to assign to $\Gamma$ a class in $R^*( \oM_{g,n}) \otimes_{\bQ} \bQ[T,T^{-1}][[Q]]$. Let $b_\Gamma$ be the boundary map $\oM_{g_0,n+|E|}\times \prod_{i=1}^{|E|} \oM_{g_i,1}\to \oM_{g,n}$. As before, the $n+|E|$ marked points on the first irreducible component are ordered such that the first $n$ marked points correspond to the legs and the last $|E|$ to the edges/nodes, respectively, preserving their given ordering. We decorate the vertices and half-edges of $\Gamma$ by $\psi$-, $\kappa$-, and $\lambda$-classes in the following way:
\begin{itemize}
	\item The component corresponding to the root vertex is endowed with the class
	\begin{equation}\label{eq:root-contribution}
		\mathcal{V}\coloneqq (T-Q)^{2g_0-2+n+|E|} \Big( 1+\sum_{m=1}^\infty\frac{(-1)^m(T-Q)^m}{m!}\sum_{i_1,\dots,i_m\ge1}V_{i_1+1}\dots V_{i_m+1}\,\kappa_{i_1,\dots,i_m}\Big),
	\end{equation}
	which can alternatively be written as
	\begin{equation}\begin{gathered}
			(T-Q)^{2g_0-2+n+|E|} \exp\Big({\sum_{i=1}^\infty s_i\kappa_i}\Big)
	\end{gathered}\end{equation}
for the coefficients $s_i$, $i\geq 1$ determined by
$e^{-\sum_{i=i}^\infty s_it^i}=1+\sum_{k=1}^\infty (T-Q)V_{k+1}t^k$.
\item For each edge $e_i$, $i\in \set{|E|}$, we combine the $\psi$-classes $\psi'$ and $\psi''$ assigned to its half-edges attached to $v_0$ and $v_i$, respectively, as well as the $\lambda$-classes assigned to $v_i$, into the following single expression:
\begin{equation}\label{eq:edge-contribution}
	\mathcal{E}_i\coloneqq \sum_{k,l=0}^\infty\sum_{a=0}^{g_i-1} (-1)^{g_i+l+1} T^{g_i-a} V_{k+l+2}\; (\psi')^{k}\otimes(\psi'')^{l} \lambda_{g_i}\lambda_{a}.
\end{equation}
(here $\psi'$ turns into $\psi_{n+i}$ on the component corresponding to $v_0$ and $\psi''$ is the $\psi_1$ on the component corresponding to $v_i$, $i=1,\dots,|E|$).
\end{itemize}
The coefficients $V_k$ are given by~\eqref{eq:Vk} which we remind here for convenience:
\begin{equation}\label{eq:Vk-reminder}
	V_k=\sum_{d=1}^\infty X^d \frac{d^{d-1}}{d!}\Bigl(\frac{d}{T}\Bigr)^k
   =\Bigl(\frac{Q}{T-Q}\frac{d}{dQ}\Bigr)^k\frac{Q}{T},
   \quad X=\frac{Q}{T}e^{-\frac{Q}{T}}.
\end{equation}
This formula implies that $V_k$ is expressed as a rational function in~$T$ and $Q$, moreover, $V_k\in \mathbb{Q}[T,(T-Q)^{-1}]$ for $k\ge2$ and $V_{1} = \frac{Q}{T(T-Q)}$. In an exceptional case $n=0$ of the formula below we will need also the functions
\begin{equation}\label{eq:Vkl-reminder}
	V_{k,l}=\sum_{d_1,d_2=1}^\infty X^{d_1+d_2} \frac{T^{-1}d_1^{d_1}d_2^{d_2}}{(d_1+d_2)d_1!d_2!}
   \Bigl(\frac{d_1}{T}\Bigr)^k\Bigl(\frac{d_2}{T}\Bigr)^l
   =\Bigl(\frac{Q}{T-Q}\frac{d}{dQ}\Bigr)^{-1}V_{k+1}V_{l+1},
\end{equation}
where by $(Q/(T-Q)\partial_Q)^{-1}$ we mean formal integration without the constant term, that is, the operator $\int_0^Q \frac{T-Q}{Q}(\cdot)dQ$.

\begin{proposition} \label{prop:materialization} We have:
	\begin{align}\label{eq:MainFormula-Cgn}
\notag  C_{g,n}(T,Q) & = \sum_{\Gamma\in\SRT_{g,n}} \frac{1}{|E|!} (b_\Gamma)_* \Big(\mathcal{V}\prod_{i=1}^{|E|}\mathcal{E}_i\Big)
		%\\ \notag & \qquad
		+ \delta_{n,1}\sum_{a=0}^{g-1} (-1)^{g} T^{g-a}{\lambda_g \lambda_a} \sum_{l=0}^\infty (-1)^{l+1}V_{l+1}\, {\psi_1^{l}}  %|_{\oM_{g,1}}.
		\\ & \quad
		+\delta_{n,0}\sum_{a=0}^{g-1} (-1)^{g} T^{g-a}{\lambda_g \lambda_a}
% \\\notag &\qquad\qquad
%\times
\Biggl(\sum_{l=0}^\infty (-1)^{l}
\bigl(\tfrac{V_{l+1}}{T-Q}-V_{1,l}\bigr)\, \kappa_{l}  %|_{\oM_{g,0}}
	\\ \notag & \qquad\qquad\qquad \qquad\qquad\qquad\qquad
	+	\frac1{2} \sum_{k,l=0}^\infty
				 (-1)^{k+l}V_{k+1,l+1}
				b_* \bigl((\psi')^{k}\otimes (\psi'')^{l}\bigr)\Biggr).
	\end{align}
%
%	\begin{align}
%		C_{g,n}(T,Q) & = \sum_{\Gamma\in\SRT_{g,n}} \frac{1}{|E|!} (b_\Gamma)_* \Big(\mathcal{V}\prod_{i=1}^{|E|}\mathcal{E}_i\Big)
%		%\\ \notag & \qquad
%		+\delta_{n,1} \sum_{l=0}^\infty \sum_{a=0}^{g-1} (-1)^{g+l+1} T^{g-a} V_{l+1}\, {\psi_1^{l}} {\lambda_g \lambda_a} %|_{\oM_{g,1}}.
%		\\ \notag & \qquad
%		+\delta_{n,0}\sum_{l=0}^\infty \sum_{a=0}^{g-1} (-1)^{g+l} T^{g-a}
%((T-Q)^{-1}V_{l+1}-V_{1,l})\, \kappa_{l} {\lambda_g \lambda_a} %|_{\oM_{g,0}}
%		\\ \notag & \qquad
%		+ \delta_{n,0} \frac1{2} \sum_{\substack{g_1+g_2=g\\g_1,g_2\geq 1}} \sum_{l_1,l_2=0}^\infty \sum_{a_1=0}^{g_1-1}\sum_{a_2=0}^{g_2-1}
%				\\ \notag & \qquad \qquad
%				 (-1)^{g+l_1+l_2} T^{g-a_1-a_2}V_{l_1+1,l_2+1}
%				(b_{g_1,g_2})_* \Big(\psi_1^{l_1} \lambda_{g_1}\lambda_{a_1}\otimes \psi_1^{l_2}\lambda_{g_2}\lambda_{a_2}\Big).
%	\end{align}
	For the exceptional terms we use here $\kappa_0|_{\oM_{g,0}} = 2g-2$; $b$ is the boundary morphism providing a $2:1$ parameterization of the boundary divisor, and $\psi',\psi''$ are the $\psi$ classes associated with the two branches of a curve at its singular point.
\end{proposition}

\begin{proof} It is convenient to abuse the notation and to think of $p_* \sum_{d=1}^\infty (-1)^{g+1}Q^d C_{g,n,d}(T)$ rather than of $\sum_{d=1}^\infty (-1)^{g+1} Q^d p_* C_{g,n,d}(T)$, though it is not a single push-forward but rather a system of push-forwards that depends on $d$. This allows us to take the sum over all labels $d\colon E\to \mathbb{Z}_{\geq 1}$ for each given pre-stable star rooted tree, with no condition on the sum of $d(e)$ over $e\in E$. For the class~\eqref{eq:C(T)-formula} rewritten in this way we assemble the dependence on $Q$ for each edge $e$ of genus $g(e)=h$
	%(where we include in the edge data the corresponding $\psi$-class on the root component and the class on moduli space corresponding to the attached non-root vertex)
	into the series
\begin{equation}\label{eq:edge}
\begin{gathered}
\sum_{d=1}^\infty \frac{d^{d-1}}{d!} \Big(\frac{Q}{T}\Big)^d
	 \sum_{k,l=0}^\infty \sum_{a=0}^{h-1} (-1)^{h+l+1} T^{h-a}\Bigl(\frac{d}{T}\Bigr)^{k+l+2}
(\psi')^{k} \otimes (\psi'')^l\lambda_h \lambda_a |_{\oM_{h,1}},\quad h\ge1,
	 \\
-\sum_{d=1}^\infty \frac{d^{d-1}}{d!} \Big(\frac{Q}{T}\Big)^d
 \sum_{k=0}^\infty \Bigl(\frac{d}{T}\Bigr)^{k}
	   (\psi')^{k},\quad h=0.
\end{gathered}
\end{equation}
The remaining dependence on $T$ is the overall factor $ T^{2g_0-2+n+|E|}$ assigned to the root vertex.

Now we need to apply the push-forward to~$\oM_{g,n}$. Under this push-forward, we have to contract the edges that connect the root vertex to the unstable vertices of genus $0$ collected in the second summand in~\eqref{eq:edge} and to forget the corresponding points. We do it in two steps. On the first one we use the string and dilaton equation to eliminate all contributions as in the summands with $k=0$ and $k=1$ on the second line. Following~\cite[Appendix A.3]{Janda-P1} and the computations of the previous section, we conclude
that the effect of the string equation amounts to the formal substitution $Q/T \mapsto Q/Te^{-Q/T}$. Note that if $X=Q/Te^{-Q/T}$, then $\sum_{d=1}^\infty \frac{d^{d-1}}{d!} X^d=Q/T$ and $T^{-1}X\partial_X=Q/(T-Q)\partial_Q$. The dilaton equation multiplies the root vertex by $(1-Q/T)^{2g_0-2+n+|E|}$. Thus the combined effect of the string and the dilaton equation implies the following intermediate form of the edge contributions under the push-forward:
\begin{equation}\label{eq:edge-string}
\begin{gathered}
	 \sum_{k,l=0}^\infty \sum_{a=0}^{h-1} (-1)^{h+l+1} T^{h-a}
%\Big(\frac{Q}{T-Q}\partial_{Q}\Big)^{k+l+2} \frac{Q}{T}
V_{k+l+2}~ (\psi')^{k} \otimes (\psi'')^l\lambda_h \lambda_a |_{\oM_{h,1}},\quad h\ge1,
	 \\
- \sum_{k=2}^\infty
%\Big(\frac{Q}{T-Q}\partial_{Q}\Big)^{k} \frac{Q}{T}
    V_k~ (\psi')^{k},\quad h=0,
\end{gathered}
\end{equation}
where $V_k$ is the series in $X=Q/Te^{-Q/T}$ given by~\eqref{eq:Vk-reminder}, while the root vertex is further equipped with the coefficient $ (T-Q)^{2g_0-2+n+|E|}$.

At the next step we forget all edges with unstable vertices, which does not change the first line in~\eqref{eq:edge-string} but converts the second line into the following contribution to the root vertex expressed in terms of the kappa classes on the moduli space corresponding to the root vertex:% (cf.~\eqref{eq:forget-kappa}):
\begin{align}
	1+\sum_{m=1}^\infty \frac{(-1)^m (T-Q)^m}{m!} \sum_{a_1,\dots,a_m\geq 1} \kappa_{a_1,\dots,a_m} \prod_{i=1}^m  V_{a_i+1}.
\end{align}

Note that there is an exceptional case when within the push-forward the moduli space corresponding to the root vertex becomes unstable (or it was unstable initially). In this case, either $n=0$ or $n=1$.  If $n=1$, the above computation gives the following class on $\oM_{g,1}$:
\begin{align} \label{eq:exceptional}
	& \sum_{l=0}^\infty \sum_{a=0}^{g-1} (-1)^{g+l+1} T^{g-a}
 V_{l+1}~
 \psi_1^{l} \lambda_g \lambda_a .
\end{align}

If $n=0$, then we have to handle the cases of graphs in $\SRT_{g,0}(\bP^1,d)$ where $g_0=0$ and at most two non-root vertices have positive genus. In this case, we can use directly the materialization equations for the unstable cases, see~\cite[Equations (A.6) and (A.7)]{Janda-P1}, which give the exceptional $\delta_{n,0}$ terms in the form
\begin{multline}
		\sum_{l=0}^\infty \sum_{a=0}^{g-1} (-1)^{g+l} T^{g-a}
\bigl(\tfrac{V_{l+1}}{T-Q}-V_{1,l}\bigr)\; \kappa_{l} {\lambda_g \lambda_a} %|_{\oM_{g,0}}
		+  \frac1{2} \sum_{\substack{g_1+g_2=g\\g_1,g_2\geq 1}} \sum_{l_1,l_2=0}^\infty \sum_{a_1=0}^{g_1-1}\sum_{a_2=0}^{g_2-1}
				\\
				 (-1)^{g+l_1+l_2} T^{g-a_1-a_2}V_{l_1+1,l_2+1}~
				(b_{g_1,g_2})_* \Big(\psi_1^{l_1} \lambda_{g_1}\lambda_{a_1}\otimes \psi_1^{l_2}\lambda_{g_2}\lambda_{a_2}\Big).
\end{multline}
where $b_{g_1,g_2}:\oM_{g_1,1}\times\oM_{g_2,1}\to\oM_{g,0}$ is the boundary morphism and $V_{k,l}$ is a series in $X=Q/Te^{-Q/T}$ given by~\eqref{eq:Vkl-reminder}. This expression is equivalent to that one of~\eqref{eq:MainFormula-Cgn} by the behavior of the $\lambda$ classes under restriction to the boundary strata.

Alternatively, one can use the following trick: exactly the same graphs but with one extra leg attached are also present in $\SRT_{g,0}(\bP^1,d)$, and the coefficients are proportional with the factor of $d$ (the latter observation is a corollary of the string equation, $\int_{\oM_{0,|E|}} \prod_{i=1}^{|E|} (1-d_i\psi_i)^{-1}  /  \int_{\oM_{0,1+|E|}} \prod_{i=1}^{|E|} (1-d_i\psi_{1+i})^{-1} = 1/d$). Thus, we can obtain the coefficients for these exceptional cases by applying $(Q\partial_Q)^{-1}$, which gives the same answer as a direct application of the materialization equations for the unstable cases.

An observation on the equivalence of the functions $V_k$ and $V_{k,l}$ represented as series in $X=Q/Te^{-Q/T}$ to their representations as rational functions in~$T$ and~$Q$ completes the proof.
\end{proof}

\subsection{Tautological relations} \label{sec:taut}
Denote by $C^{(p)}_{g,n}(T,Q)$ the homogeneous component of (cohomological) degree $p$ in the class $C_{g,n}(T,Q)$, that is, $C^{(p)}_{g,n}(T,Q)\in R^{p}(\oM_{g,n})[T,Q]$. By construction, it is a polynomial in $T$ and $Q$ of homogeneous degree~$2g-2+n-p$. In particular, it identically vanishes if $p > 2g-2+n$. The goal of this section is to analyze the tautological relations that these vanishings generate.

\subsubsection{Coefficients of classes}
We use notation $S\coloneqq T-Q$. The dependence of the terms entering the main formula~\eqref{eq:MainFormula-Cgn} on $T$ and $Q$ is through the functions $V_k$ and $V_{k,l}$ whose definitions in terms of $T$ and $S$ reads
\begin{align}
V_k&=\Bigl(\frac{S-T}{S}\partial_S\Bigr)^k\frac{T-S}{T},\\
V_{k,l}&=\Bigl(\frac{S-T}{S}\partial_S\Bigr)^{-1}V_{k+1}V_{l+1},
\end{align}
where the integration constant in the last formula is chosen such that $V_{k,l}|_{S=T}=0$.

\begin{lemma}
All the functions $V_k$, $k\ge1$, and $V_{k,l}$, $k,l\ge1$ are Laurent polynomials in~$T$ and~$S$ of homogeneous degrees $-k$ and $-k-l-1$, respectively.
\end{lemma}

\begin{corollary}
The contribution of each graph in $\SRT_{g,n}$ and each exceptional summand in~\eqref{eq:MainFormula-Cgn} to $C_{g,n}^{(p)}(T,Q)$ is a Laurent polynomial in~$T$ and~$S$ of homogeneous degree~$2g-2+n-p$.
\end{corollary}

The Laurent polynomials in $T$ and $S=T-Q$ should be thought of as infinite power series in~$Q$. However, since the total polynomial $C_{g,n}^{(p)}(T,Q)$ is represented as a finite sum of Laurent polynomials, the total contribution of each monomial in~$T$ and~$S$ which is of negative degree in either~$T$ or~$S$ should vanish, hence, it provides a tautological relation in $R^{p}(\cM_{g,n})$. In particular, if $p>2g-2+n$, then the vanishing of the total contribution of each monomial gives a tautological relation.

In fact, with certain exception for $n=0$, the monomials with negative exponent of $T$ never appear in computations, as the following lemma shows.

\begin{lemma}\label{lem:T-regular}
We have:
\begin{itemize}
\item $V_1=\frac{1}{S}-\frac{1}{T}$ and for $k\ge2$ the Laurent polynomial $V_k$ is regular in~$T$, that is,
\begin{equation}
V_k\in\bQ[S^{-1},T],\quad k\ge2.
\end{equation}
Moreover,
\begin{equation}\label{eq:Vk-at-T=0}
V_k\bigm|_{T=0}=(-1)^{k-1}\frac{(k-1)!}{S^k},\quad k\ge2.
\end{equation}
\item $V_{1,0}=\frac{1}{2T^2}-\frac{1}{ST}+\frac{1}{2S^2}$, and for $k,l\ge1$, up to one explicitly given monomial in $T$, the function $V_{k,l}$ is also regular in~$T$:
\begin{equation}\label{eq:Vkl-at-T=0}
V_{k,l}-\frac{(-1)^{k+1}B_{k+l}}{(k+l)T^{k+l+1}}\in\bQ[S^{\pm1},T],\quad k,l\ge1.
\end{equation}
\end{itemize}
Here $B_l$, $l\geq 0$, are the Bernoulli numbers given by
\begin{align}
	\sum_{l=0}^\infty B_l \frac{t^l }{l!}= \frac{te^t}{e^t-1}.
\end{align}
\end{lemma}

The proof of both lemmas are straightforward elementary computations. \qed

Using the last lemma, one can extract some particular coefficients of monomials in~$T$ and $S=T-Q$ in~\eqref{eq:MainFormula-Cgn}. Our special interest is the coefficient of $T^0S^d$ for different $d$ (both positive and negative).

\begin{lemma}\label{lem:T=0contribution} Substitute $Q=T-S$ in the expression for $C_{g,n}(T,Q)$ given by~\eqref{eq:MainFormula-Cgn}. Then
for any $d\in\Z$ the coefficient of the monomial $T^0S^d$ in the resulting expression is equal to
\begin{equation}\label{eq:T=0contribution}
J_{2g-2+n-d}+\delta_{n,1}\delta_{d,0}(-1)^{g}\lambda_g\lambda_{g-1}
+\delta_{n,0}\delta_{d,0}(-1)^{g-1}\lambda_g\lambda_{g-2}\in R^{2g-2+n-d}(\oM_{g,n}),
\end{equation}
where the $J_p$ are the polynomials in~$\kappa$ classes defined by~\eqref{eq:J-first}--\eqref{eq:J-class}.
\end{lemma}

\begin{proof}
Since $V_k$ is regular in~$T$ for $k\ge2$, all summands in the main term in~\eqref{eq:MainFormula-Cgn} are also regular in~$T$. Moreover, since the exponent $g_i-a$ in~\eqref{eq:edge-contribution} is strictly positive, all graphs in $\SRT_{g,n}$ with a nonempty edge set provide trivial contributions to the coefficient of $T^0$. The only graph that provides a nontrivial contribution is the graph with a single vertex of genus $g$ and no edges. By~\eqref{eq:root-contribution}, where we substitute the values of $V_k$ given in~\eqref{eq:Vk-at-T=0}, the contribution of this graph is equal exactly to the corresponding~$J$-class, as presented in~\eqref{eq:J-class}.

By the same reason, a nontrivial contribution of the exceptional terms for~$n=1$ in~\eqref{eq:MainFormula-Cgn} to the coefficient of $T^0$ is possible for $l=0$ and $a=g-1$ only, and this contribution is equal to~$(-1)^{g}\lambda_g\lambda_{g-1}$.

In opposite to the cases considered above, the exceptional summand with $n=0$ in~\eqref{eq:MainFormula-Cgn} is not regular in~$T$, so we need a more careful computation here. Using~\eqref{eq:Vkl-at-T=0} and ignoring the terms regular in~$T$ in the third line in~\eqref{eq:MainFormula-Cgn}, we compute the contribution of the exceptional summand with~$n=0$ to the coefficient of $T^0$ equal to
\begin{multline}\label{eq:n=0proof1}
[T^0]\sum_{a=0}^{g-1} (-1)^{g} T^{g-a}{\lambda_g \lambda_a}
\biggl(-\frac{\kappa_0}{2T^2}+\sum_{l=1}^\infty
\frac{B_{l+1}}{(l+1)T^{l+2}}\, \kappa_{l}%|_{\oM_{g,0}}
%		\\ \notag & \qquad\qquad\qquad
		\biggr.
\\
+\biggl.\frac1{2} \sum_{k,l=0}^\infty
				 (-1)^{k}\frac{B_{k+l+2}}{(k+l+2)T^{k+l+3}}
				b_* \bigl((\psi')^{k}\otimes (\psi'')^{l}\bigr)\biggr)
\\=(1-g)(-1)^g\lambda_g\lambda_{g-2}+[T^0]\sum_{a=0}^{g-1} (-1)^{g} T^{g-a-2}{\lambda_g \lambda_a}
\sum_{k=1}^\infty\frac{B_{2k}}{2k}\frac{\mu_{2k-1}}{T^{2k-1}},
\end{multline}
where the classes $\mu_l$ for odd~$l$ are defined by
\begin{equation}
\mu_l=\kappa_l+\frac12b_*\frac{(\psi')^l+(\psi'')^l}{\psi'+\psi''}.
\end{equation}
Recall that by Mumford's formula~\cite{Mumford}, we have
\begin{equation}
e^{\sum_{k=1}^\infty\frac{B_{2k}}{(2k-1)2k}\mu_{2k-1}t^{2k-1}}=\sum_{a=0}^g\lambda_at^a.
\end{equation}
Applying $t\partial_t$ we get as a corollary the identity
\begin{equation}
\biggl(\sum_{k=1}^\infty\frac{B_{2k}}{2k}\mu_{2k-1}t^{2k-1}\biggr)
\sum_{a=0}^g\lambda_at^a=\sum_{a=0}^ga\,\lambda_at^a.
\end{equation}
By this identity,~\eqref{eq:n=0proof1} can be rewritten as
\begin{multline}
(1-g)(-1)^g\lambda_g\lambda_{g-2}+[T^0](-1)^{g} T^{g-2}\lambda_g
\sum_{a=0}^ga\,\frac{\lambda_a}{T^a}\\
=(1-g)(-1)^g\lambda_g\lambda_{g-2}+(g-2)(-1)^g\lambda_g\lambda_{g-2}=-(-1)^{g}\lambda_g\lambda_{g-2},
\end{multline}
which agrees with~\eqref{eq:T=0contribution}.
\end{proof}

\subsubsection{Proofs of the main theorems}
For $d<0$ the equality of Lemma~\ref{lem:T=0contribution} is a tautological relation: the total coefficient of $T^0S^d$ in the expression for the class $C_{g,n}(T,Q)$ for negative $d$ must be equal to zero.

If $d=0$, then the monomial $T^0S^d$ is a constant. Hence, the total contribution to this monomial should also vanish since the class $C_{g,n}(T,Q)$ does not contain the constant terms in the expansion in~$Q$, by definition.

This gives immediately the proofs of Theorems~\ref{thm:main} and \ref{thm:n10}.

In the case $d>0$, Lemma~\ref{lem:T=0contribution} is also applicable. In this case, it provides not a tautological relation any more but rather an explicit formula for the corresponding term of $C_{g,n}$ immediately implying the $d>0$ part of the statement of Theorem~\ref{thm:C=J}.

Finally, for $d=0$ case of Theorem~\ref{thm:C=J}, we compute the classes explicitly without applying localization, just directly from its definition (see~\cite{KL,KiemLi} and~\cite[Section 1.1]{giacchetto2024spingromovwittenhurwitzcorrespondencemathbbp1}). To this end, we use the following trick: we consider $(\bP,\cO(-1))$ as the normal bundle of the exceptional divisor $\tilde X$ in $\tilde Y$ obtained as a blow-up at one point of a K3 surface $Y$. Then we can just apply the computation of~\cite[Section 2.2]{GetPan}. We have $c_1(\tilde Y) = -[\tilde X]$, and
$[\oM_{g,n}(\tilde Y,0)]^{\mathrm{vir}} = [\oM_{g,n}
\times \tilde Y] \cap (\lambda_g\lambda_{g-1} [\tilde X] + \lambda_{g}\lambda_{g-2} [\tilde X]^2)$. Then for $p\colon \oM_{g,n}
\times \tilde Y \to \oM_{g,n}$ we have
\begin{align} \label{eq:d0-comuptations-v1}
	(-1)^{g+1} p_* \Big( [\oM_{g,1}
	\times \tilde Y] \cap (\lambda_g\lambda_{g-1} [\tilde X] + \lambda_{g}\lambda_{g-2} [\tilde X]^2) \cap ev_1^*(-[\tilde X]) \Big) & = (-1)^{g-1} \lambda_{g} \lambda_{g-1};
	\\ \notag
	(-1)^{g+1} p_*  \Big( [\oM_{g,0}
	\times \tilde Y] \cap (\lambda_g\lambda_{g-1} [\tilde X] + \lambda_{g}\lambda_{g-2} [\tilde X]^2) \Big) & = (-1)^{g} \lambda_g \lambda_{g-2}.
\end{align}
\qed

\begin{remark}
	Alternatively, one can use~\cite[Theorem 1.1]{KiemLi} in combination with~\cite[Equation~(2.6)]{KiemLi}. This theorem states that in the case of $d=0$ we have to replace the class $c_{g-1+d} (R^1\pi_*f^*\cO(-1))$ in Equation~\eqref{eq:MainClass} by the class $-c_{g-1} (R^1\pi_*f^*\cO(-1) - R^0\pi_* f^*\cO(1))$. On $\oM_{g,n}(\bP^1,0) = \oM_{g,n}\times \bP^1$ the latter class is equal to $(-1)^{g-1}(\lambda_{g-1}\otimes 1 + 3 \lambda_{g-2} \otimes [\mathrm{pt}])$. Recall that $[ \oM_{g,n}(\bP^1,0)]^{\mathrm{vir}} = (-1)^{g-1} (\lambda_{g}\otimes 1 - 2 \lambda_{g-1}\otimes [\mathrm{pt}])$~\cite[Section 2.1]{GetPan} and in the case $n=1$ we have $ev_1^*([\mathrm{pt}])= 1\otimes [\mathrm{pt}]$. Then for $p\colon \oM_{g,n}\times \bP^1\to \oM_{g,n}$, for $n=0,1$, we have:
	\begin{align}
		 (-1)^{g+1} p_* \Big([\oM_{g,1}\times \bP^1]\cap
		(1\otimes [\mathrm{pt}]) \cdot (-1)^{g-1} (\lambda_{g}\otimes 1 - 2 \lambda_{g-1}\otimes [\mathrm{pt}]) \cdot \quad & \\ \notag   (-1)^{g-1}(\lambda_{g-1}\otimes 1 + 3 \lambda_{g-2} \otimes [\mathrm{pt}])
		\Big)  & = (-1)^{g-1} \lambda_g\lambda_{g-1};
		\\ \notag
		(-1)^{g+1} p_* \Big([\oM_{g,0}\times \bP^1]\cap
		 (-1)^{g-1} (\lambda_{g}\otimes 1 - 2 \lambda_{g-1}\otimes [\mathrm{pt}]) \cdot \quad & \\ \notag   (-1)^{g-1}(\lambda_{g-1}\otimes 1 + 3 \lambda_{g-2} \otimes [\mathrm{pt}])
		 & = (-1)^g \lambda_g\lambda_{g-2}
	\end{align}
	as we have already checked in Equation~\eqref{eq:d0-comuptations-v1}.
\end{remark}

\section{Equivariant topological recursion}

The purpose of this section is to explain where the computations of the previous section are originated from. The original conjecture on kappa classes in~\cite{KN-conjecture} is linked in \emph{op.~cit.} to a regularity question of a system of differentials obtained by topological recursion of Chekhov-Eynard-Orantin~\cite{EO-TR}. In this section we give a reformulation of the setup for topological recursion proposed in~\cite{KN-conjecture} in terms of the so-called $\mathbb{Z}_2$-equivariant topological recursion introduced in~\cite{GKL}.

We prove that this $\mathbb{Z}_2$-equivariant topological recursion describes the stationary sector of the spin Gromov-Witten theory of $\bP^1$, thus establishing the regularity proposed in~\cite[Section 2.2.1]{KN-conjecture}. This proof is in fact a numerical avatar of the argument in the previous section that we presented in Section~\ref{sec:numerical}.

\subsection{Topological recursion and its equivariant version}
First, let us recall the definition of the original topological recursion of Chekhov-Eynard-Orantin \cite{CEO,EO-TR}. It associates a system of meromorphic differentials $\omega^{(g)}_n$ (also known as $n$-point differentials), $g\geq 0$, $n\geq 1$, $2g-2+n\geq 0$, to an input data (the \emph{spectral curve data}) that consists of a Riemann surface $\Sigma$ and a finite set of points $\cP\subset \Sigma$, two functions $x$ and $y$ on $\Sigma$  such that $dx$ and $dy$ are meromorphic differentials on $\Sigma$ with $q\in\cP$ being {simple} critical points of $x$, and $\restr {} {q} {dy}\not= 0$ for each $q\in \cP$, and a bi-differential $B$ with a double pole on the diagonal with biresidue $1$ and no other poles. %In the original formulation of the spectral curve topological recursion the points from $\cP$ are required to be %\emph
%{simple} critical points of $x$.
%It has its origin in the computation of the cumulants of the matrix models, and by now it has multiple striking applications in algebraic geometry, enumerative combinatorics, and mathematical physics.

The symmetric differentials $\omega^{(g)}_{n}$, $g\geq 0$, $n\geq 1$, are constructed as follows:
\begin{align}\label{eq:om0102}
	\omega_{0,1}(z_1) = y(z_1) dx(z_1) ; \qquad	
	\omega_{0,2}(z_1,z_2) = B(z_1,z_2);
\end{align}
and for $2g-2+n>0$ the following recursion is used:
\begin{align}\label{eq:TopologicalRecursion}
	\omega_{g,n} (z_1,\dots,z_n) \coloneqq \ & \frac 12 \sum_{\xi\in\cP} \res_{z\to \xi}
	\frac{\int_z^{\sigma_\xi(z)} B(z_1,\cdot)} {(y(\sigma_\xi(z)) -  y(z))\,dx(z) }\Bigg(
	\omega_{g-1,n+1}(z,\sigma_\xi(z),z_{\llbracket n \rrbracket \setminus \{1\}})
	\\
	\notag
	& + \sum_{\substack{g_1+g_2 = g, I_1\sqcup I_2 = {\llbracket n \rrbracket \setminus \{1\}} \\
			(g_1,|I_1|),(g_2,|I_2|) \not= (0,0) }} \omega_{g_1,1+|I_1|}(z,z_{I_1})\omega_{g_2,1+|I_2|}(\sigma_\xi(z), z_{I_2})\Bigg),
\end{align}
where  $\sigma_\xi$ is the deck transformation of $x$ near $\xi$ (that is, it is a local involution such that $x(\sigma_\xi(z))=x(z)$) and $z_I$ denotes $\{z_i\}_{i\in I}$ for $I\subseteq \set{n}$.

Now recall the definition of equivariant topological recursion introduced by Giacchetto-Kramer-Lewa\'nski \cite[Definition~8.1]{GKL}:
\begin{definition}\label{defn:equiv:SC}
	Let $G$ be a finite group. A \emph{$G$-equivariant spectral curve data}
\begin{align}
	%\mathcal{S} =
	(\Sigma,\phi, x,y,B, \chi, \upsilon, \beta)	
\end{align}	
 consists of
	\begin{itemize}
		\item a Riemann surface $\Sigma$ (not necessarily compact or connected) with a free action $ \phi \colon G \times \Sigma \to \Sigma $, for which the following notation is used: $\phi (\gamma,z) = \phi_\gamma z = \gamma z$, $\gamma\in G$, $z\in \Sigma$;
		\item a function $x \colon \Sigma \to \C$ such that its differential $dx$ is meromorphic and has finitely many zeros $a_1,\dots,a_r$ that are simple;
		\item a meromorphic function $y \colon \Sigma \to \C$ such that $d y$ does not vanish at the zeros of $d x$;
		\item a symmetric bidifferential $B$ on $\Sigma \times \Sigma$;
		\item three one-dimensional representations $ \chi, \upsilon, \beta \colon G \to \C^*$,
	such that for any $ \gamma \in G$
	\begin{equation}\label{Eqdefs}
		dx (\gamma z) = \chi_\gamma \, dx (z) \,, \qquad dy (\gamma z) = \upsilon_\gamma \, dy(z)\,, \qquad B(\gamma z_1,z_2) = \beta_\gamma \, B(z_1,z_2)\,,
	\end{equation}
	and $ B(z_1,z_2) - B^{G,\beta}(z_1,z_2) $ is regular as $ z_1 \to \gamma z_2$, where
	\begin{equation}\label{eq:Bgb}
		B^{G,\beta}(z_1,z_2) \coloneqq %\frac{1}{|G|}
		\sum_{\eta \in G} \beta_\eta^{-1} \frac{d (\eta z_1) dz_2}{(\eta z_1 - z_2)^2}\,.
	\end{equation}
		\end{itemize}
	This version of  %\emph
	{topological recursion} constructs the $n$-point differentials via the usual formula~\eqref{eq:TopologicalRecursion}, whereas for the unstable cases it is natural to define them as
	\begin{align}
		\omega^{(0)}_1(z_1)&  = \frac{1}{|G|} y(z_1) dx(z_1); & \omega^{(0)}_2(z_1,z_2) & = B(z_1,z_2).
	\end{align}
\end{definition}
\begin{remark}
	Definition~\ref{defn:equiv:SC} as given here differs from the one given in \cite{GKL} by the absence of the factor $|G|^{-1}$ in formula~\eqref{eq:Bgb} as compared to the respective formula in~\cite[Definition~8.1]{GKL}. This factor amounts to a simple rescaling of $n$-point differentials by some factors of $2$ and its absence or presence does not affect any properties; for our purposes it is more convenient not to have it.
	
	We also slightly modified \eqref{Eqdefs}, namely, we formulate it as a condition on the differential $dy$ comparing to the condition on the function $y$ in \cite[Definition~8.1]{GKL}. This definition is slightly more general, but due to the way $y$ enters~\eqref{eq:TopologicalRecursion}  all statements regarding equivariant topological recursion hold for a curve with such a relaxed assumption. We need this relaxed condition in the example below.
\end{remark}

An example of the spectral curve data for the original topological recursion that is of primary interest for this paper is given in \cite[Section~2.2.1]{KN-conjecture} (and has its roots in~\cite{Eynard-Intersection}):
\begin{align}\label{eq:SigmaKN}
	\Sigma&=\mathbb{P}^1;
	& x&=\frac{w^2}{2};
	& y&=\frac{2\arcsinh(w/\sqrt{2\epsilon})}{\sqrt{w^2+2\epsilon}};
	&	B  &=\frac{dw_1dw_2}{(w_1-w_2)^2}. %\label{eq:BKN}
\end{align}

An example of a $G$-equivariant spectral curve data is given in Theorem~\ref{thm:TR-spin}. It is a $\mathbb{Z}_2$-equivariant spectral curve data according to Definition~\ref{defn:equiv:SC}. Namely, we have
\begin{align}\label{eq:GZ2}
G&=\mathbb{Z}_2; \quad
 \Sigma=\mathbb{P}^1; \quad
 \iota(z)=-\frac{1}{z};\\
  x&=\sqrt{\epsilon/2} \left(z+\frac 1z\right); \quad
y=2\log z; \quad
B  =\frac{dz_1dz_2}{(z_1-z_2)^2} + \frac{d\iota(z_1)dz_2}{(\iota(z_1)-z_2)^2},
\end{align}
where $\iota$ denotes the action of the generator of $\mathbb{Z}_2$ (that is $\iota (z) = \phi(1,z)$, where $\mathbb{Z}_2$ is $\{0,1\}$ with 0 being the group unit); and we have replaced $Q$ with $-\epsilon$. For this $\mathbb{Z}_2$-equivariant spectral curve $\chi=\upsilon$ is the sign representation and $\beta$ is the trivial representation of $\mathbb{Z}_2$.

Now we want to compare the $n$-point differentials produced by $\mathbb{Z}_2$-equivariant topological recursion on the curve~\eqref{eq:GZ2} with the ones produced by topological recursion on the curve~\eqref{eq:SigmaKN}. In order not to get confused, let us mark all objects associated with~\eqref{eq:SigmaKN} with tilde (that is, for that curve we will write $\tilde \Sigma$, $\tilde x$, $\tilde y$, $\tilde B$, $\tilde \omega^{(g)}_n$), while keeping the notation for~\eqref{eq:GZ2} (with the $n$-point differentials produced by $\mathbb{Z}_2$-equivariant topological recursion on~\eqref{eq:GZ2} denoted by $\omega^{(g)}_n$).
%Let $\{\tilde \omega^{(g)}_n\}$, respectively $\{\omega^{(g)}_n\}$, be the system of differentials produced by ($\mathbb{Z}_2$-equivariant) topological recursion for the spectral curve data~\eqref{eq:SigmaKN}, respectively~\eqref{eq:GZ2}.
%The following proposition compares these two setups of topological recursion.

The following proposition implements this comparison:
\begin{proposition} The systems of differentials $\{\tilde \omega^{(g)}_n\}$ and $\{\omega^{(g)}_n\}$ are related by the following formula:
\begin{align}
\prod_{i=1}^n \restr{w_i}{\sqrt{\epsilon/2}(z_i-z_i^{-1})}	2^{2g-2+n} \tilde \omega^{(g)}_n(w_1,\dots,w_n) = \omega^{(g)}_n(z_1,\dots,z_n).
\end{align}
\end{proposition}
\begin{proof}
	Consider $\tilde yd\tilde x$ %$\omega^{(0)}_1$
	(that is, $ydx$ for the spectral curve \eqref{eq:SigmaKN}):
	\begin{equation}
		\tilde y d \tilde x %\tilde \omega^{(0)}_1
		= \frac{2\arcsinh(w/\sqrt{2\epsilon})}{\sqrt{w^2+2\epsilon}}  d\frac{w^2}{2},
	\end{equation}	
	where $w$ is a global rational coordinate on $\tilde\Sigma=\bP^1$.
	
	Introduce a new variable $z$ related to $w$ by $w=\sqrt{\epsilon/2}(z-z^{-1})$; %. Note that
	$z$ is a global rational coordinate on the double cover of $\tilde\Sigma=\mathbb{P}^1$ (which is $\mathbb{P}^1$ itself; this will be $\Sigma$ of~\eqref{eq:GZ2}). In the new coordinate we have locally near the point $w=0$, $z=1$:
	\begin{equation}
		\frac{2\arcsinh(w/\sqrt{2\epsilon})}{\sqrt{w^2+2\epsilon}} d\frac{w^2}{2} = 2\log(z)\, d \left(\sqrt{{\epsilon/2}}\left(z+\frac 1z\right) \right),
	\end{equation}
	where the latter expression is equal to $ydx$, with $x$ and $y$ of~\eqref{eq:GZ2}. %$\omega^{(0)}_1$.
	
	Now notice that for $\iota:z \mapsto -\frac 1z$ we have
	\begin{equation}
		\tilde B=\frac{dw_1dw_2}{(w_1-w_2)^2} = \frac{dz_1dz_2}{(z_1-z_2)^2} + \frac{d(\iota(z_1))dp_2}{(\iota(z_1)-z_2)^2}=B,
	\end{equation}
	hence the recursion on~\eqref{eq:SigmaKN} coincides locally with the recursion on~\eqref{eq:GZ2} whose data is taken locally near $z=1$. The latter recursion is called (according to the terminology of~\cite{GKL}) the reduced version of the full recursion~\eqref{eq:GZ2} and its differentials are tautologically equal to  $\tilde \omega^{(g)}_n$ . The full equivariant version of the latter reduced recursion has an extra critical point at $z=-1$ and by~\cite[Proposition 8.8]{GKL} we have $\omega^{(g)}_n = 2^{2g-2+n} \tilde \omega^{(g)}_n$.
\end{proof}

\subsection{Expansion of the differentials} In this Section, we give a proof of Theorem~\ref{thm:TR-spin}.
\begin{proof}[Proof of Theorem~\ref{thm:TR-spin}]	
Note that by Eynard's formula~\cite{Eynard-Intersection}, as presented in~\cite{KN-conjecture}, we have for $2g-2+n>0$
\begin{align}
	\tilde\omega^{(g)}_n (w_1,\dots,w_n) = (-1)^n \sum_{k_1,\dots,k_n=0}^\infty \int_{\oM_{g,n}} \sum_{d=0}^\infty \epsilon^d J_{2g-2+n-d} \prod_{i=1}^n \psi_i^{k_i} \prod_{i=1}^n \frac{(2k_i+1)!!}{w_i^{2k_i+2}} dw_i.
\end{align}
By Theorems~\ref{thm:main} and~\ref{thm:C=J}, we have
\begin{align}
	\tilde\omega^{(g)}_n (w_1,\dots,w_n) & = (-1)^{g+d+n+1} \sum_{k_1,\dots,k_n=0}^\infty \sum_{d=0}^\infty \epsilon^d \int_{p_*C_{g,n,d}} \prod_{i=1}^n \psi_i^{k_i} \prod_{i=1}^n \frac{(2k_i+1)!!}{w_i^{2k_i+2}} dw_i.
\end{align}
This gives the stationary sector of the ancestor spin Gromov-Witten theory of $\bP^1$. Passing to descendants in the stable range $2g-2+n>0$ is given by a linear change of variables as in~\eqref{eq:ansc-desc}, and it amounts to the re-expansion of $\tilde\omega^{(g)}_n (w_1,\dots,w_n)$ in $\sqrt{w^2+2\epsilon} = x(z)$ at the point $z=0$, due to the following easily verified identity:
\begin{equation}
\frac{(2k+1)!!}{x^{2k+2}}dx=\sum_{m=0}^\infty\frac{ Q^m}{m!}\frac{(2(k+m)+1)!!}{w^{2(k+m)+2}}dw,
\end{equation}
where we identify $Q=-\epsilon$. We obtain:
\begin{align}
	\tilde\omega^{(g)}_n (w_1,\dots,w_n) & = (-1)^{g+n+1} \sum_{k_1,\dots,k_n=0}^\infty \sum_{d=0}^\infty Q^d \int_{C_{g,n,d}} \prod_{i=1}^n \psi_i^{k_i} \prod_{i=1}^n \frac{(2k_i+1)!!}{x_i^{2k_i+2}} dx_i.
\end{align}
%Now note that $\sum_{i=1}^n {k_i} = g-1+d$, hence the total exponent of $2$ is equal to $2g-2+2d+n$.
Note also that $\omega^{(g)}_n = 2^{2g-2+n} \tilde\omega^{(g)}_n$. Hence,
\begin{align}
	\omega^{(g)}_n (w_1,\dots,w_n) & = (-1)^{n} \sum_{k_1,\dots,k_n=0}^\infty \sum_{d=0}^\infty Q^d 2^{2g-2+n} (-1)^{g+1} \int_{C_{g,n,d}} \prod_{i=1}^n \psi_i^{k_i} \prod_{i=1}^n \frac{(2k_i+1)!!}{x_i^{2k_i+2}} dx_i.
\end{align}
Taking into account the extra sign $(-1)^{g+1}$ that was natural for us to include in the definition of the descendant invariants, we obtain the statement of Theorem~\ref{thm:TR-spin} in the stable range $2g-2+n>0$. It extends further to the unstable case $(0,2)$ by a straightforward computation that essentially repeats the proof of Proposition~\ref{prop:materialization}. Indeed, applying~\cite[Equation (A.6)]{Janda-P1} in our case, we get for $(g,n)=(0,2)$
\begin{align}
\sum_{k_1,k_2=0}^\infty \frac{(-1)^{k_1+k_2+1}}{\zeta_1^{k_1+1}\zeta_2^{k_2+1}}\langle  \tau_{k_1}\tau_{k_2}\rangle^{\bP^1,\cO(-1),{\rm dec}}_{0} = \frac{e^{\frac Q{\zeta_1} + \frac Q{\zeta_2}} -1}{\zeta_1+\zeta_2},
\end{align}
and then one can use this formula to check directly that
\begin{align}
	\frac{dw_1dw_2}{(w_1-w_2)^2} - \frac{dx_1dx_1}{(x_1-x_2)^2} = \sum_{k_1,k_2=0}^\infty \langle  \tau_{k_1}\tau_{k_2}\rangle^{\bP^1,\cO(-1),{\rm dec}}_{0}  \frac{(2k_1+1)!!}{x_1^{2k_1+2}}\frac{(2k_2+1)!!}{x_2^{2k_2+2}}  dx_1 dx_2
\end{align}
in expansion near $x=\infty$.
\end{proof}

\printbibliography

\end{document}